\documentclass[11pt,a4paper]{article}
\usepackage[a4paper,
            bindingoffset=0.2in,
            left=1in,
            right=1in,
            top=1in,
            bottom=1in,
            footskip=.25in]{geometry}
\usepackage{amsmath, amssymb, amsthm}
\usepackage{graphicx}
\usepackage{cite}
\usepackage{hyperref}
\usepackage{mathrsfs}
\usepackage{url}
\usepackage{algorithm}
\usepackage{algorithmic}
\usepackage{epstopdf}
\usepackage{pgfplots}
\pgfplotsset{compat=1.18} 

\hypersetup{
    colorlinks=true,
    linkcolor=blue,
    citecolor=blue,
    urlcolor=blue
}

\newtheorem{theorem}{Theorem}[section]
\newtheorem{lemma}[theorem]{Lemma}

\newtheorem{definition}[theorem]{Definition}
\newtheorem{remark}[theorem]{Remark}

\title{The Unified Separation Condition: \\ 
       A General Constraint Qualification for \\ 
       Smooth and Nonsmooth Optimization}
\author{M. E. Abbasov\thanks{abbasov.majid@gmail.com, m.abbasov@spbu.ru} \\
        St. Petersburg State University, \\ 
        7/9 Universitetskaya nab., St. Petersburg, 199034 Russia}
\date{}

\begin{document}

\maketitle

\begin{abstract}
We introduce the \emph{Unified Separation Condition (USC)} --- a simple and general constraint qualification for finite-dimensional nonlinear programming. The USC requires that the origin be uniformly separated from the subdifferential of the constraint violation function in a neighborhood of the feasible set.

The main result of this paper is two-fold. First, we show that the classical Mangasarian--Fromovitz constraint qualification (MFCQ) for inequalities and the linear independence constraint qualification (LICQ) for equalities imply the USC. Conversely, USC is strictly more general: it remains applicable in nonsmooth settings where classical conditions are not defined, and it may hold even when MFCQ fails. The proof that MFCQ implies USC is based on Gordan's theorem; the implication from LICQ to USC follows from the linear independence of the gradients of the active constraints.

We also introduce a local version of USC and discuss its computational verification via convex quadratic programming. The USC provides a unified framework for constraint qualifications, bridging classical smooth theory and nonsmooth optimization.


\end{abstract}

\section{Introduction}

Constraint qualifications are fundamental to optimization theory. They ensure that the Karush--Kuhn--Tucker (KKT) conditions are necessary for local optimality and that Lagrange multipliers exist. The earliest systematic study of constraint qualifications dates back to the 1950s, following the seminal work of Kuhn and Tucker \cite{KuhnTucker1951}, who first established necessary optimality conditions for nonlinear programming under differentiability assumptions. Their conditions, however, required some form of regularity to ensure that the Lagrange multipliers were well-defined. This need gave rise to a rich hierarchy of constraint qualifications that have been developed over the subsequent decades.

Among the most widely used constraint qualifications are the linear independence constraint qualification (LICQ) and the Mangasarian--Fromovitz constraint qualification (MFCQ) \cite{Mangasarian1967}. LICQ requires that the gradients of the active constraints are linearly independent, guaranteeing uniqueness of the Lagrange multipliers. MFCQ, introduced by Mangasarian and Fromovitz in 1967, is a weaker condition: it requires that the gradients of the active equality constraints are linearly independent and that there exists a direction which is orthogonal to all these gradients and strictly decreases all active inequality constraints. MFCQ ensures that the set of Lagrange multipliers is nonempty and bounded \cite{Gauvin1977}. The relationship between these conditions has been extensively studied; for a comprehensive overview of the hierarchy of constraint qualifications under differentiability assumptions, see the survey by Giorgi, Jiménez, and Novo \cite{Giorgi2026}. These conditions are treated separately, require differentiability of the active constraints, and have no direct connection to numerical algorithms.

An alternative line of development emerged from the theory of exact penalty functions. In their seminal papers, Eremin \cite{Eremin1967} and Zangwill \cite{Zangwill1967} introduced the notion of exact penalization, showing that a constrained optimization problem could be reformulated as an unconstrained one without loss of precision, provided the penalty parameter is chosen sufficiently large. In this theory, the constraints are incorporated into the objective function through a penalty term multiplied by a positive parameter. For sufficiently large values of this parameter, the minimizers of the penalized problem coincide with those of the original constrained problem. Such penalty functions are called exact. The connection between exact penalty functions and constraint qualifications was further explored by Burke \cite{Burke1991}, who showed how the Eremin--Zangwill exact penalty functions could be used to develop the foundations of constrained optimization theory in finite dimensions in an elementary and straightforward way. Regularity conditions, multiplier rules, second-order optimality conditions, and convex programming were all given interpretations relative to exact penalty functions. Subsequent work by Han and Mangasarian \cite{HanMangasarian1979}, Di Pillo and Grippo \cite{DiPilloGrippo1989}, and Di Pillo and Facchinei \cite{DiPillo1989} further considered exact penalty functions for nonsmooth problems. Coleman and Conn \cite{ColemanConn1982} proposed an algorithm based on an exact penalty function with global convergence properties, while Contaldi, Di Pillo, and Lucidi \cite{ContaldiDiPilloLucidi1993} introduced a continuously differentiable exact penalty function for problems with unbounded feasible sets.

A key result due to Demyanov \cite{Dem2005} gives sufficient conditions for the existence of an exact penalty constant. One of these conditions requires that, in a neighborhood of the feasible set, the constraint violation function has a uniformly negative rate of decrease. This condition is equivalently expressed as the uniform separation of the origin from the subdifferential of the constraint violation function.Demyanov's work builds on the broader framework of constructive nonsmooth analysis, which he developed extensively with Rubinov \cite{DemRub1995}. In this framework, subdifferentials are defined through directional derivatives, providing a natural language for describing the local behavior of nonsmooth functions. This framework was further extended through the theory of exhausters and coexhausters, which provide more flexible representations of directional derivatives. Abbasov and Demyanov \cite{DemyanovAbbasov_2010} derived extremum conditions for nonsmooth functions in terms of exhausters and coexhausters, including second-order approximations. Abbasov \cite{Abbasov2015} introduced generalized exhausters and established optimality conditions under weaker assumptions, while Abbasov \cite{Abbasov2017} compared quasidifferentials and exhausters, showing that the latter are preferable even for quasidifferentiable functions in the context of nonsmooth optimization algorithms. 

The theory of exact penalization in nonsmooth settings was further advanced by Demyanov, di Pillo, and Facchinei \cite{DGP1998}, who analyzed exact penalty functions using Dini and Hadamard conditional derivatives and established weak conditions guaranteeing equivalence of stationary points. Demyanov, Giannessi, and Karelin \cite{DGK1998} demonstrated the applicability of exact penalty functions to optimal control problems, while Demyanov and Tamasyan \cite{DT2011, DemTam2014} extended the approach to isoperimetric problems and parametric moving boundary variational problems.

In a separate line of research, significant progress has been made in developing constraint qualifications tailored for specific classes of nonsmooth problems. Notably, Dolgopolik \cite{Dolgopolik2019} proposed a new constraint qualification for quasidifferentiable programming. A key novelty of his approach is that the resulting optimality conditions are formulated using specific elements chosen from the quasidifferential, rather than the quasidifferential as a whole.The present paper is also related to the general theory of variational analysis and generalized differentiation developed by Mordukhovich \cite{Mordukhovich2006I, Mordukhovich2006II}, which provides a comprehensive framework for nonsmooth optimization and stability analysis, as well as to the classical nonsmooth analysis of Clarke \cite{Clarke1983}. In the context of optimal control, the present author \cite{Abbasov2026} applied USC to derive transversality conditions and Pontryagin's maximum principle, demonstrating the power of this approach in infinite-dimensional settings. The USC framework also connects naturally to recent advances in subdifferential calculus \cite{HantouteKrugerLopez2026} and duality theory for convex infinite optimization \cite{GobernaVolle2022}.

The present paper identifies Demyanov's technical condition as a general constraint qualification for finite-dimensional nonlinear programming. We call it the \emph{Unified Separation Condition (USC)}. We prove that MFCQ for inequalities and LICQ for equalities are special cases of USC, and demonstrate that USC naturally extends to nonsmooth problems where classical conditions are not defined. Unlike classical constraint qualifications, USC does not require differentiability of the constraint functions and can be monitored computationally during optimization. A local version of USC is also introduced, making it practical for verifying optimality conditions and analyzing numerical algorithms.

Thus, the USC provides a unified framework that bridges classical smooth constraint qualifications and nonsmooth optimization, while remaining computationally verifiable.

The paper is organized as follows. Section 2 presents the problem formulation and recalls necessary preliminaries from exact penalty theory and constructive nonsmooth analysis. Section 3 introduces the USC and establishes its connection to Demyanov's condition and exact penalty functions. Section 4 proves that MFCQ and LICQ imply USC for inequality and equality constraints, respectively, and discusses the mixed case. Section 5 introduces a local version of USC. Section 6 provides illustrative examples. Section 7 discusses computational verification. Section 8 concludes the paper.

\section{Problem Formulation and Preliminaries}

\subsection{Problem Statement}

Consider the standard nonlinear programming problem:
\begin{equation}
    \min_{x \in \mathbb{R}^n} f(x)
    \label{eq:NLP}
\end{equation}
subject to
\begin{equation}
    g_i(x) \le 0, \quad i \in I = \{1, \dots, m\},
    \label{eq:ineq_constraints}
\end{equation}
where $f$ and $g_i$ are functions from $\mathbb{R}^n$ to $\mathbb{R}$. We denote the feasible set by
\begin{equation}
    \Omega = \{ x \in \mathbb{R}^n \mid g_i(x) \le 0, \; i \in I \}.
\end{equation}

For simplicity of exposition, we focus initially on inequality constraints. Equality constraints and mixed constraints will be discussed in Section~\ref{sec:mixed}. The functions $f$ and $g_i$ are assumed to be locally Lipschitz (a standard assumption in nonsmooth analysis). In the smooth case, they are assumed to be continuously differentiable.

\subsection{Exact Penalty Functions}

The theory of exact penalty functions \cite{Eremin1967, Zangwill1967, Dem2005} provides a powerful framework for converting constrained optimization problems into unconstrained ones. Define the constraint violation function
\begin{equation}
    \phi(x) = \max\{0, g_1(x), \dots, g_m(x)\}.
    \label{eq:phi_def}
\end{equation}

This function has the following properties:
\begin{itemize}
    \item $\phi(x) \ge 0$ for all $x \in \mathbb{R}^n$.
    \item $\phi(x) = 0$ if and only if $x \in \Omega$.
    \item $\phi$ is locally Lipschitz if all $g_i$ are locally Lipschitz.
    \item $\phi$ is a nonsmooth function even when the $g_i$ are smooth.
\end{itemize}

The exact penalty function is defined as
\begin{equation}
    F_\lambda(x) = f(x) + \lambda \phi(x),
    \label{eq:penalty}
\end{equation}
where $\lambda \ge 0$ is the penalty parameter.

\begin{definition}[Exact Penalty Constant]
\label{def:exact_penalty}
A number $\lambda^* \ge 0$ is called an \emph{exact penalty constant} for problem \eqref{eq:NLP}--\eqref{eq:ineq_constraints} if for all $\lambda > \lambda^*$,
\begin{equation}
    \inf_{x \in \mathbb{R}^n} F_\lambda(x) = \inf_{x \in \Omega} f(x),
\end{equation}
and any local minimizer of $F_\lambda$ is a local minimizer of the original problem.
\end{definition}

\begin{figure}[h]
  \centering
  \begin{minipage}{0.45\textwidth}
    \centering
    \includegraphics[width=0.95\textwidth]{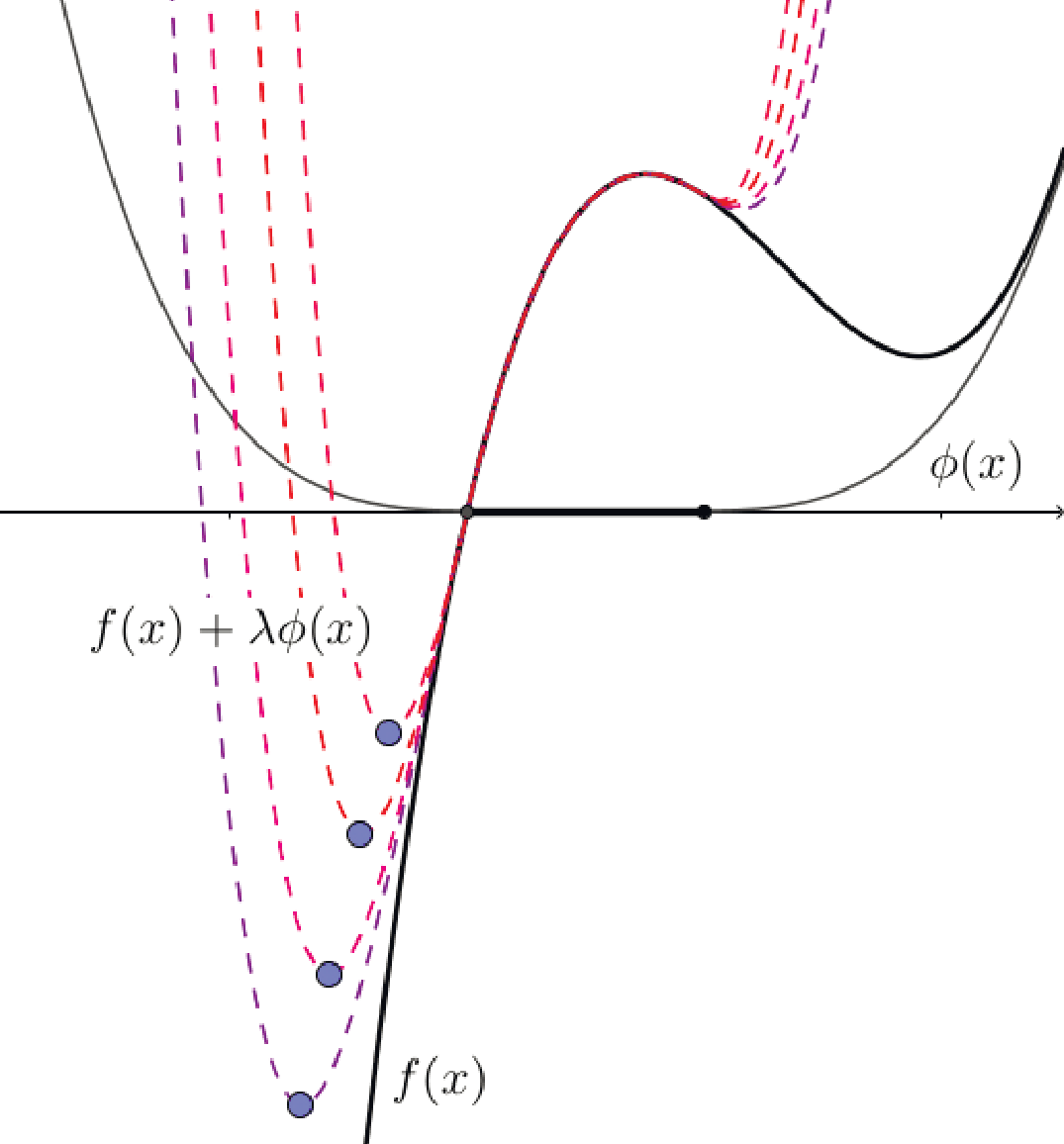}
    \small Smooth penalty
  \end{minipage}\hfill
  \begin{minipage}{0.45\textwidth}
    \centering
    \includegraphics[width=0.95\textwidth]{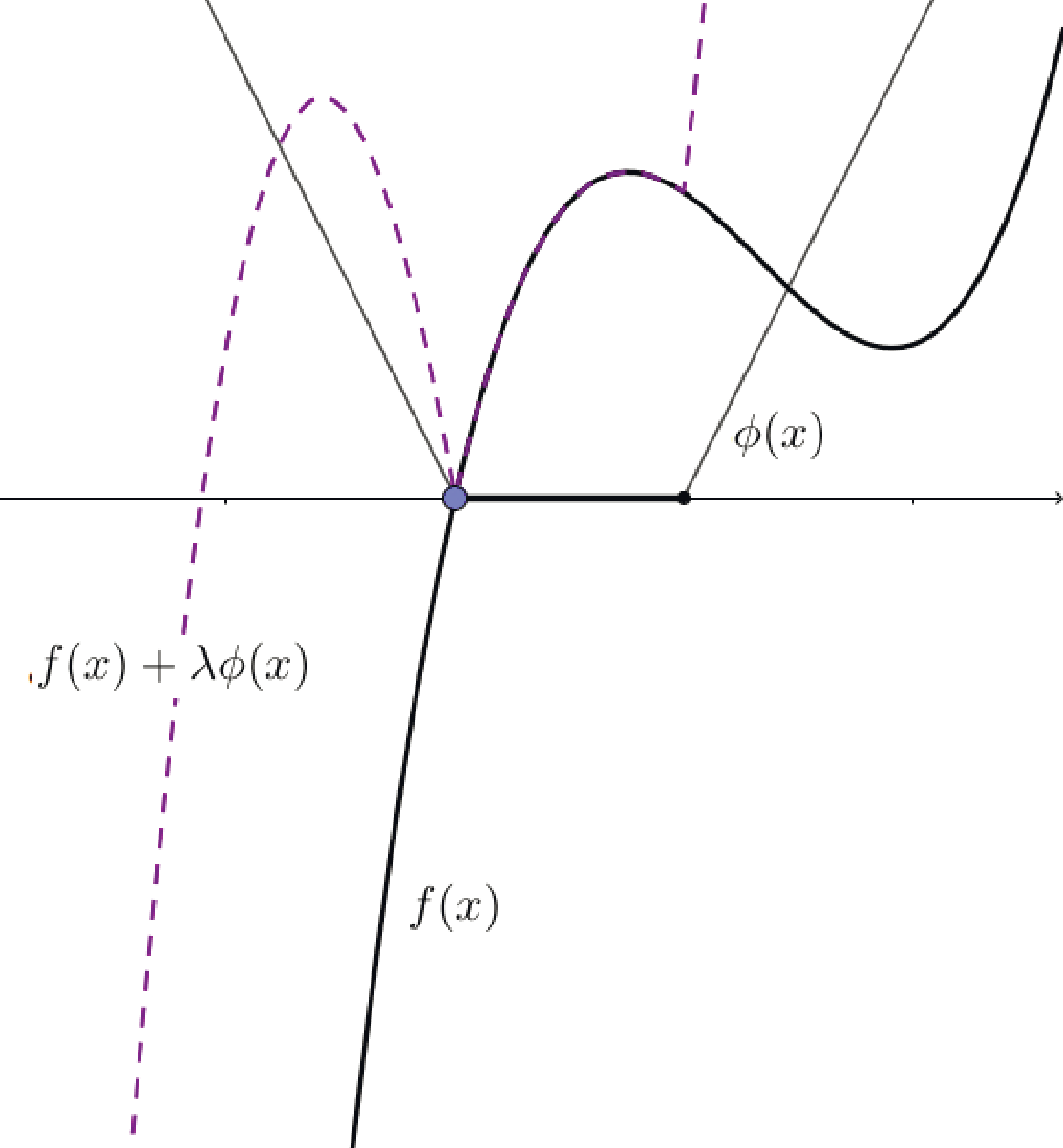}
    \small Nonsmooth exact penalty
  \end{minipage}
  \caption{Classical smooth penalties flatten at the feasible set $\Omega$, while nonsmooth exact penalties retain a nonzero first-order slope. This kink is the geometric reason why exact feasibility can be obtained for a finite penalty parameter.}
  \label{fig:smooth-vs-exact-penalty}
\end{figure}

The fundamental mechanism behind exactness is simple but important. The decisive property is not merely that $\phi$ is small near $\Omega$, but that it has a nonzero first-order rate of change with respect to infeasibility. A clean model case is a single inequality constraint
\[
\Omega = \{x \in \mathbb{R}^n \mid g(x) \le 0\},
\]
where $g$ is continuously differentiable in a neighborhood of the boundary $\{g=0\}$ and is nondegenerate there, say $\|\nabla g(x)\| \ge c > 0$. Then the nonsmooth penalty
\[
\phi(x) = \max\{0, g(x)\}
\]
has a nonzero first-order slope when approaching the feasible set from the infeasible side $g(x) > 0$. This is the relevant exact-penalty mechanism; by contrast, a smooth quadratic violation penalty such as $\max\{0, g(x)\}^2$ has zero first derivative at the boundary. This geometric difference is illustrated in Fig.~\ref{fig:smooth-vs-exact-penalty}. Consequently, a finite multiplier $\lambda$ can make the nonsmooth penalty dominate the possible first-order decrease of $f$ outside $\Omega$. For smooth quadratic penalties, exact feasibility is typically recovered only asymptotically as $\lambda \to \infty$; increasing $\lambda$ also makes the penalized problem increasingly ill-conditioned, which is one of the main numerical drawbacks of classical smooth penalty methods.

A convenient way to express this nonzero-slope property is through the lower derivative
\begin{equation}
\phi^\downarrow(x) = \liminf_{y \to x} \frac{\phi(y) - \phi(x)}{\|y - x\|}.
\end{equation}
When $\phi^\downarrow(x) \le -a < 0$ uniformly in a neighborhood of the feasible set, the penalty function has a guaranteed direction of descent at a rate bounded away from zero. This allows a finite penalty parameter to dominate the objective function $f$ and thus enforce exact feasibility.

\begin{theorem}[Demyanov's Exact Penalty Theorem]
\label{thm:demyanov}
Let $\Omega = \{x \in X \mid \phi(x) = 0\}$, where $\phi(x) \ge 0$. Suppose that:
\begin{enumerate}
    \item $\displaystyle\inf_{x \in X} f(x) > -\infty$.
    \item There exists $\lambda_0 < \infty$ such that for each $\lambda \ge \lambda_0$, there is $x_\lambda \in X$ attaining the infimum of $F_\lambda(x) = f(x) + \lambda \phi(x)$.
    \item There exist $\delta > 0$ and $a > 0$ such that
    \begin{equation}
        \phi^\downarrow(x)\le -a < 0 \quad \forall x \in \Omega_\delta \setminus \Omega,
        \label{eq:demyanov_condition}
    \end{equation}
    where $\Omega_\delta = \{x \in X \mid \phi(x) < \delta\}$.
    \item $f$ is Lipschitz on $\Omega_\delta \setminus \Omega$.
\end{enumerate}
Then there exists $\lambda^* \ge \lambda_0$ such that for all $\lambda > \lambda^*$, $F_\lambda$ is an exact penalty function.
\end{theorem}

\begin{theorem}[Equivalence of Local Minima]
\label{thm:local_equivalence}
Let the conditions of Theorem \ref{thm:demyanov} hold, and let $x^* \in \Omega$ be a local minimizer of $f$ on $\Omega$. Then there exists $\lambda^* < \infty$ such that for all $\lambda > \lambda^*$, in a sufficiently small neighborhood $B$ of $x^*$:
\begin{enumerate}
    \item[(i)] Any local minimizer of $F_\lambda$ on $B$ is a local minimizer of $f$ on $\Omega$.
    \item[(ii)] The point $x^*$ is a local minimizer of $F_\lambda$ on $B$.
\end{enumerate}
\end{theorem}

This theorem establishes the equivalence between local minima of the original constrained problem and those of the penalized unconstrained problem in a neighborhood of the solution. The key condition enabling this equivalence is (3): in a neighborhood of the feasible set, the constraint violation function must have a uniformly negative rate of decrease. Demyanov notes in \cite{Dem2005} that this condition is equivalent to requiring that the subdifferential of $\phi$ is uniformly bounded away from zero.

\subsection{Subdifferential Calculus}

We briefly recall some concepts from constructive nonsmooth analysis \cite{Dem2005, Demyanov1995}. For a locally Lipschitz function $\phi: \mathbb{R}^n \to \mathbb{R}$, the subdifferential at $x$ in the sense of Demyanov is defined through the directional derivative:
\begin{equation}
    \phi'(x; d) = \lim_{\alpha \downarrow 0} \frac{\phi(x + \alpha d) - \phi(x)}{\alpha}.
\end{equation}
If the directional derivative exists and has the form
\begin{equation}
    \phi'(x; d) = \max_{v \in \partial \phi(x)} \langle v, d \rangle,
\end{equation}
then the set $\partial \phi(x)$ is called the subdifferential of $\phi$ at $x$. For the functions considered in this paper, which are maxima of smooth functions, this subdifferential is always well-defined and is a convex compact set.

For the constraint violation function $\phi(x) = \max\{0, g_1(x), \dots, g_m(x)\}$, when $x \notin \Omega$ (i.e., $\phi(x) > 0$), the subdifferential is given by
\begin{equation}
    \partial \phi(x) = \operatorname{conv}\{ \nabla g_i(x) \mid i \in R(x) \},
    \label{eq:subdiff_phi}
\end{equation}
where
\begin{equation}
    R(x) = \{ i \in I \mid g_i(x) = \phi(x) \}
\end{equation}
is the set of indices of active constraints.

\begin{remark}
When $x \in \Omega$ (i.e., $\phi(x) = 0$), the subdifferential includes the zero term:
\begin{equation}
    \partial \phi(x) = \operatorname{conv}\{ 0, \nabla g_i(x) \mid i \in I_0(x) \},
\end{equation}
where $I_0(x) = \{ i \in I \mid g_i(x) = 0 \}$.
\end{remark}

\section{The Unified Separation Condition}

\subsection{From Demyanov's Condition to USC}

We now establish a key link between Demyanov's condition \eqref{eq:demyanov_condition} and the subdifferential. The following result from nonsmooth analysis plays a central role in our approach: it allows us to reformulate the purely analytic condition $\phi^\downarrow(x) \le -a < 0$ in geometric terms, as the uniform separation of the origin from the subdifferential. This reformulation is essential for establishing the connection with classical constraint qualifications. For the reader's convenience, we provide a complete proof of this result for the class of functions relevant to our setting.

\begin{lemma}[Lower Derivative via Subdifferential]
\label{lem:lower_derivative}
Let $\phi: \mathbb{R}^n \to \mathbb{R}$ be a locally Lipschitz function that is the maximum of a finite collection of smooth functions. Then for every $x \in \mathbb{R}^n$,
\begin{equation}
    \phi^\downarrow(x) = -\operatorname{dist}(0, \partial \phi(x)).
\end{equation}
\end{lemma}

\begin{proof}
By definition,
\begin{equation}
    \phi^\downarrow(x) = \liminf_{y \to x} \frac{\phi(y) - \phi(x)}{\|y - x\|}.
\end{equation}
Since $\phi$ is the maximum of smooth functions, the directional derivative exists in every direction $d$ and is given by
\begin{equation}
    \phi'(x; d) = \max_{v \in \partial \phi(x)} \langle v, d \rangle.
\end{equation}

We first prove that
\begin{equation}
    \phi^\downarrow(x) = \inf_{\|d\| = 1} \phi'(x; d).
\end{equation}

\textit{Step 1:} $\phi^\downarrow(x) \ge \inf_{\|d\| = 1} \phi'(x; d)$.

Take any sequence $y_k \to x$, $y_k \neq x$. Set $\alpha_k = \|y_k - x\|$ and $d_k = (y_k - x)/\|y_k - x\|$. Then $\alpha_k \to 0$, $\|d_k\| = 1$, and
\begin{equation}
    \frac{\phi(y_k) - \phi(x)}{\|y_k - x\|} = \frac{\phi(x + \alpha_k d_k) - \phi(x)}{\alpha_k}.
\end{equation}
For any $\varepsilon > 0$ and sufficiently large $k$, by the definition of the directional derivative,
\begin{equation}
    \frac{\phi(x + \alpha_k d_k) - \phi(x)}{\alpha_k} \ge \phi'(x; d_k) - \varepsilon \ge \inf_{\|d\| = 1} \phi'(x; d) - \varepsilon.
\end{equation}
Taking the limit inferior and letting $\varepsilon \to 0$, we obtain
\begin{equation}
    \phi^\downarrow(x) \ge \inf_{\|d\| = 1} \phi'(x; d).
\end{equation}

\textit{Step 2:} $\phi^\downarrow(x) \le \inf_{\|d\| = 1} \phi'(x; d)$.

Let $\{d_k\}$ be a sequence of unit vectors such that $\phi'(x; d_k) \to \inf_{\|d\|=1} \phi'(x; d)$. 
For each $k$, by the definition of the directional derivative, there exists $\alpha_k > 0$ such that
\begin{equation}
    \frac{\phi(x + \alpha_k d_k) - \phi(x)}{\alpha_k} < \phi'(x; d_k) + \frac{1}{k}.
\end{equation}
Without loss of generality, assume that $\alpha_k < \alpha_{k-1}/2$. Then $\alpha_k \to 0$ and, taking $y_k = x + \alpha_k d_k$, we have $y_k \to x$ and
\begin{equation}
    \phi^\downarrow(x) \le \liminf_{k \to \infty} \frac{\phi(y_k) - \phi(x)}{\|y_k - x\|} \le \liminf_{k \to \infty} \left( \phi'(x; d_k) + \frac{1}{k} \right) = \inf_{\|d\|=1} \phi'(x; d).
\end{equation}
Thus the equality is proved.

Now, by the minimax theorem applied to the bilinear function $\langle v, d \rangle$ on the convex compact set $\partial \phi(x)$ and the compact unit sphere,
\begin{equation}
    \min_{\|d\| = 1} \max_{v \in \partial \phi(x)} \langle v, d \rangle = \max_{v \in \partial \phi(x)} \min_{\|d\| = 1} \langle v, d \rangle.
\end{equation}
For any fixed $v$, the minimum over $\|d\| = 1$ is attained at $d = -v/\|v\|$ (if $v \neq 0$), giving
\begin{equation}
    \min_{\|d\| = 1} \langle v, d \rangle = -\|v\|.
\end{equation}
Therefore,
\begin{equation}
    \phi^\downarrow(x) = \max_{v \in \partial \phi(x)} (-\|v\|) = -\min_{v \in \partial \phi(x)} \|v\| = -\operatorname{dist}(0, \partial \phi(x)).
\end{equation}
This completes the proof.
\end{proof}

Thus, Demyanov's condition $\phi^\downarrow(x) \le -a < 0$ is equivalent to the uniform separation of the origin from the subdifferential:
\begin{equation}
    \operatorname{dist}(0, \partial \phi(x)) \ge a > 0 \quad \forall x \in \Omega_\delta \setminus \Omega.
\end{equation}

We formalize this as follows:

\begin{definition}[Unified Separation Condition, USC]
\label{def:usc}
We say that the \emph{Unified Separation Condition} holds at a feasible point $x^* \in \Omega$ if there exist a neighborhood $B$ of $x^*$ and a constant $a > 0$ such that
\begin{equation}
    \operatorname{dist}\bigl(0, \partial \phi(x)\bigr) \ge a > 0
    \quad \forall x \in B \setminus \Omega,
    \label{eq:usc_def}
\end{equation}
where $\operatorname{dist}(0, A) = \inf_{v \in A} \|v\|_2$.
\end{definition}

\begin{remark}
The USC has a simple geometric interpretation: in a neighborhood of the feasible set, the subdifferential of the constraint violation function $\phi$ must be uniformly bounded away from zero. This means that every point outside the feasible set has a direction of descent for $\phi$, and the rate of descent is uniformly bounded below.
\end{remark}

\subsection{USC and Exact Penalty Functions}

The USC is precisely the condition that guarantees the existence of an exact penalty constant. The following theorem is a direct consequence of Demyanov's results \cite{Dem2005} (see also \cite{Abbasov2026}).

\begin{theorem}[Exact Penalty via USC]
\label{thm:exact_penalty}
Suppose that:
\begin{enumerate}
    \item The USC holds at $x^* \in \Omega$.
    \item $f$ is locally Lipschitz in a neighborhood of $x^*$.
\end{enumerate}
Then there exists $\lambda^* < \infty$ such that for all $\lambda > \lambda^*$, $x^*$ is a local minimizer of the penalty function $F_\lambda(x) = f(x) + \lambda \phi(x)$ on $\mathbb{R}^n$. Moreover, any local minimizer of $F_\lambda$ for $\lambda > \lambda^*$ is a local minimizer of the original problem.
\end{theorem}

This theorem follows directly from Theorems \ref{thm:demyanov} and \ref{thm:local_equivalence}. Indeed, by Lemma \ref{lem:lower_derivative}, the USC is precisely condition (3) of Theorem \ref{thm:demyanov} reformulated in terms of the subdifferential. Conditions (1) and (2) of Theorem \ref{thm:demyanov} are standard technical assumptions that are satisfied in our setting when the result is applied locally: condition (1) holds because we only need the infimum of $f$ on the compact neighborhood where the USC is verified; condition (2) holds because $F_\lambda$ attains its infimum on compact sets. Alternatively, one can apply the local versions of these theorems directly, as stated in Theorem \ref{thm:local_equivalence}. Applying Theorems \ref{thm:demyanov} and \ref{thm:local_equivalence} yields the result.

Thus, the USC is not merely a technical condition for exact penalization -- it is a fundamental regularity condition that guarantees the equivalence between the constrained and unconstrained problems.

\section{USC Implies Classical Constraint Qualifications}

In this section, we establish the main theoretical contribution of this paper: we show for the first time that the condition on the lower derivative, originally introduced by Demyanov as a sufficient condition for exact penalization, is in fact a unified constraint qualification that encompasses the classical LICQ and MFCQ conditions as special cases.

\subsection{The Case of Inequality Constraints}

We begin by considering the case where only inequality constraints are present. For a point $x \in \mathbb{R}^n$, define the set of indices of active constraints:
\begin{equation}
    R(x) = \{ i \in I \mid g_i(x) = 0 \}.
\end{equation}

We shall use the following classical result (see \cite{Gordan1873}).

\begin{theorem}[Gordan's Theorem]
\label{thm:gordan}
For a finite set of vectors $\{a_1, \dots, a_r\} \subset \mathbb{R}^n$, exactly one of the following alternatives holds:
\begin{enumerate}
    \item[(A)] There exists a vector $d \in \mathbb{R}^n$ such that $\langle a_i, d \rangle < 0$ for all $i = 1, \dots, r$.
    \item[(B)] There exist coefficients $\lambda_i \ge 0$, not all zero, such that $\sum_{i=1}^r \lambda_i a_i = 0$.
\end{enumerate}
\end{theorem}

Recall the classical MFCQ for inequality constraints:

\begin{definition}[Mangasarian--Fromovitz Constraint Qualification]
\label{def:mfcq}
The MFCQ holds at $x^* \in \Omega$ if there exists a direction $d \in \mathbb{R}^n$ such that
\begin{equation}
    \langle \nabla g_i(x^*), d \rangle < 0 \quad \forall i \in R(x^*).
    \label{eq:mfcq_def}
\end{equation}
\end{definition}

We now prove the main equivalence result for inequality constraints.

\begin{theorem}[MFCQ implies USC]
\label{thm:mfcq_implies_usc}
Let $x^* \in \Omega$ be a feasible point. If the MFCQ holds at $x^*$, then the USC holds at $x^*$.
\end{theorem}

\begin{proof}
Assume MFCQ holds. Then there exists a unit vector $d$ such that $\langle \nabla g_i(x^*), d \rangle < 0$ for all $i \in R(x^*)$. Define
\begin{equation}
    a = -\frac{1}{2} \max_{i \in R(x^*)} \langle \nabla g_i(x^*), d \rangle > 0.
\end{equation}

By continuity, there exists a neighborhood $B$ of $x^*$ such that for all $x \in B$ and all $i \in R(x^*)$,
\begin{equation}
    \langle \nabla g_i(x), d \rangle \le -a.
\end{equation}
Shrinking $B$ if necessary, we may assume that any constraint inactive at $x^*$ remains negative in $B$, so $R(x) \subseteq R(x^*)$ for all $x \in B$.

Now take any $x \in B \setminus \Omega$. Then $\phi(x) > 0$, and
\begin{equation}
    \partial \phi(x) = \operatorname{conv}\{ \nabla g_i(x) \mid i \in R(x) \}.
\end{equation}
For any $v \in \partial \phi(x)$, there exist coefficients $\alpha_i \ge 0$ for $i \in R(x)$ such that $\sum_{i \in R(x)} \alpha_i = 1$ and $v = \sum_{i \in R(x)} \alpha_i \nabla g_i(x)$. Then
\begin{equation}
    \langle v, d \rangle = \sum_{i \in R(x)} \alpha_i \langle \nabla g_i(x), d \rangle \le -a.
\end{equation}
Since $\|d\|_2 = 1$, we have $\|v\|_2 \ge |\langle v, d \rangle| \ge a$. Therefore,
\begin{equation}
    \operatorname{dist}(0, \partial \phi(x)) = \min_{v \in \partial \phi(x)} \|v\|_2 \ge a > 0.
\end{equation}
Thus USC holds.
\end{proof}

\subsection{The Case of Equality Constraints}

Now consider the case of equality constraints:
\begin{equation}
    h_j(x) = 0, \quad j \in E = \{1, \dots, p\}.
    \label{eq:eq_constraints}
\end{equation}

We define the constraint violation function as
\begin{equation}
    \phi(x) = \max_{j \in E} |h_j(x)|.
    \label{eq:phi_eq}
\end{equation}

\begin{definition}[Linear Independence Constraint Qualification]
\label{def:licq}
The LICQ holds at $x^*$ if the gradients $\{\nabla h_j(x^*) \mid j \in E\}$ are linearly independent.
\end{definition}

\begin{theorem}[LICQ Implies USC for Equality Constraints]
\label{thm:licq_usc}
Let $x^*$ be a feasible point for the equality constraints (so $\phi(x^*) = 0$) and assume LICQ holds at $x^*$. Then the USC holds at $x^*$.
\end{theorem}

\begin{proof}
For $x \notin \Omega$, we have $\phi(x) > 0$. Let
\begin{equation}
    E_{\text{act}}(x) = \{ j \in E \mid |h_j(x)| = \phi(x) \}.
\end{equation}
The subdifferential of $\phi$ at $x$ is
\begin{equation}
    \partial \phi(x) = \operatorname{conv}\{ \operatorname{sign}(h_j(x)) \nabla h_j(x)  \mid j \in E_{\text{act}}(x) \}.
\end{equation}

Since $E_{\text{act}}(x) \subseteq E$, there are only finitely many possible active sets. For each nonempty subset $S \subseteq E$, consider the function
\begin{equation}
    d_S(x) = \operatorname{dist}\left(0, \operatorname{conv}\{ \operatorname{sign}(h_j(x)) \nabla h_j(x) \mid j \in S \}\right).
\end{equation}
In a sufficiently small neighborhood of $x^*$, the signs of $h_j(x)$ are constant for each $j \in E$ (we can choose the neighborhood so that no $h_j$ changes sign except possibly at $x^*$ itself). Hence each $d_S$ is continuous on $B \setminus \Omega$. Since the gradients $\{\nabla h_j(x^*)\}_{j \in E}$ are linearly independent by LICQ, any subset of them is also linearly independent, and hence for every nonempty $S \subseteq E$,
\begin{equation}
    d_S(x^*) > 0.
\end{equation}
By continuity, there exists a neighborhood $B_S$ of $x^*$ such that $d_S(x) \ge d_S(x^*)/2 > 0$ for all $x \in B_S \setminus \Omega$. Taking 
\begin{equation}
    B = \bigcap_{\substack{S \subseteq E \\ S \neq \emptyset}} B_S,
\end{equation}
we obtain a uniform constant $a =\displaystyle \min_{S \subseteq E, S \neq \emptyset} d_S(x^*)/2 > 0$ such that for all $x \in B \setminus \Omega$,
\begin{equation}
    \operatorname{dist}(0, \partial \phi(x)) \ge a > 0.
\end{equation}
Thus USC holds.
\end{proof}

\subsection{Mixed Equality and Inequality Constraints}
\label{sec:mixed}

The USC for mixed constraints requires only that the convex hull of the gradients of all active constraints (both equalities and inequalities) is uniformly separated from zero. This is precisely the same geometric condition as in the pure cases, showing that USC treats all constraints uniformly.

Define
\begin{equation}
    \phi(x) = \max\left\{ \max_{j \in E} |h_j(x)|, \max_{i \in I} g_i(x), 0 \right\}.
    \label{eq:phi_mixed}
\end{equation}

For $x \notin \Omega$ (i.e., $\phi(x) > 0$), the subdifferential of $\phi$ at $x$ is
\begin{equation}
    \partial \phi(x) = \operatorname{conv}\left( \left\{ \operatorname{sign}(h_j(x)) \nabla h_j(x) \mid j \in E_{\text{act}}(x) \right\} \cup \left\{ \nabla g_i(x) \mid i \in I_{\text{act}}(x) \right\} \right),
\end{equation}
where
\begin{equation}
    E_{\text{act}}(x) = \{ j \in E \mid |h_j(x)| = \phi(x) \}, \qquad
    I_{\text{act}}(x) = \{ i \in I \mid g_i(x) = \phi(x) \}.
\end{equation}

The USC requires that the distance from the origin to $\partial \phi(x)$ is uniformly bounded below by $a > 0$ for all $x \in B \setminus \Omega$ in a neighborhood $B$ of $x^*$.

The classical regularity condition for mixed constraints is:
\begin{enumerate}
    \item The gradients $\{\nabla h_j(x^*)\}_{j \in E}$ are linearly independent.
    \item There exists a direction $d$ such that $\langle \nabla h_j(x^*), d \rangle = 0$ for all $j \in E$ and $\langle \nabla g_i(x^*), d \rangle < 0$ for all $i \in I_{\text{act}}(x^*)$.
\end{enumerate}

This condition implies USC by the same arguments used in the pure equality and inequality cases: the linear independence of the equality gradients ensures uniform separation for the equality part, while the existence of a common descent direction on the nullspace of the equality gradients ensures uniform separation for the inequality part.

Conversely, USC provides a unified generalization that does not require the linear independence of equality gradients or the existence of a common descent direction for inequalities; it only requires that the origin be uniformly separated from the convex hull of the active constraint gradients.

\section{Local Application of USC}

A crucial feature of USC is its local nature. In many practical situations, we are only interested in a neighborhood of a candidate solution $x^*$. The following result shows that we can apply USC locally without requiring it to hold globally.

\begin{theorem}[Local USC]
\label{thm:local_usc}
Let $x^* \in \Omega$ be a feasible point. Suppose there exist a neighborhood $B$ of $x^*$ and a constant $a > 0$ such that
\begin{equation}
    \operatorname{dist}(0, \partial \phi(x)) \ge a \quad \forall x \in B \setminus \Omega.
    \label{eq:local_usc_def}
\end{equation}
Then $x^*$ is a local minimizer of $f$ on $\Omega$ if and only if it is a local minimizer of the penalty function $F_\lambda$ on $B$ for sufficiently large $\lambda > 0$.
\end{theorem}

\begin{proof}
The proof follows from Theorem~\ref{thm:exact_penalty} applied to the localized set $\Omega \cap B$. The key point is that USC only needs to hold in $B$; points outside $B$ are irrelevant for local optimality.
\end{proof}

This local version of USC is particularly useful in practice. It requires verifying the separation condition only in a neighborhood of the solution, rather than on the entire feasible set. The penalty parameter can be chosen adaptively based on the local geometry, and the framework provides a natural setting for analyzing algorithms that generate sequences converging to a local solution. From a computational standpoint, verifying USC locally only requires solving convex optimization problems in a small region of interest, which is significantly more efficient than global verification.

\section{Illustrative Examples}

\subsection{Example: Nonsmooth Constraint and Local Reformulation}

Consider the problem:
\begin{equation}
    \min_{x \in \mathbb{R}} f(x) = x^2
\end{equation}
subject to
\begin{equation}
    g(x) = \min\{\max(0, x, -x-2), \max(-x+2, x-4)\} \le 0.
\end{equation}

\begin{figure}[ht]
\centering
\begin{tikzpicture}
\begin{axis}[
    width=11cm, height=6cm,
    axis lines=middle,
    xlabel={$x$},
    ylabel={$g(x)$},
    xmin=-3.5, xmax=5.5,
    ymin=-1.4, ymax=2.6,
]

\addplot[black, thick, domain=-4.5:-2, samples=2] {-x-2};
\addplot[black, thick, domain=-2:0, samples=2] {0};
\addplot[black, thick, domain=0:1, samples=2] {x};
\addplot[black, thick, domain=1:3, samples=2] {-x+2};
\addplot[black, thick, domain=3:6.5, samples=2] {x-4};

\addplot[line width=2pt, black] coordinates {(-2,0) (0,0)};
\addplot[line width=2pt, black] coordinates {(2,0) (4,0)};

\node[below, font=\small] at (axis cs:-1.14, 0.6) {$[-2, 0]$};
\node[below, font=\small] at (axis cs:3, 0.6) {$[2, 4]$};

\end{axis}
\end{tikzpicture}
\caption{The graph of the function $g(x)$ and the feasible set $\Omega = [-2, 0] \cup [2, 4]$.}
\label{AM_fig:feasible_set}
\end{figure}
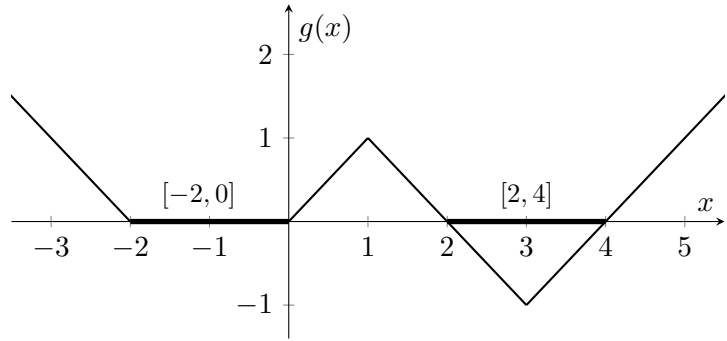

The feasible set is $\Omega = [-2, 0] \cup [2, 4]$ and is illustrated in Figure \ref{AM_fig:feasible_set}. The minimum of $f(x) = x^2$ on $\Omega$ is attained at $x^* = 0$. At this point, the constraint is active: $g(0) = 0$. The function $g$ is not differentiable at $x^* = 0$ (the left derivative is $0$, the right derivative is $1$), so MFCQ is not defined in the nonsmooth formulation.

One could, of course, analyze the structure of $g$ and observe that locally near $x^* = 0$ the constraint is equivalent to $x \le 0$, for which classical MFCQ holds. However, this requires a separate analysis of the constraint structure. The USC approach avoids this: it works directly with the nonsmooth constraint violation function
\begin{equation}
    \phi(x) = \max\{0, g(x)\}.
\end{equation}
For $x > 0$ sufficiently close to $0$, we have $\partial \phi(x) = \{1\}$ and $\operatorname{dist}(0, \partial \phi(x)) = 1$. For $x < 0$, the interval $[-2, 0]$ is contained in $\Omega$, so there are no points outside $\Omega$ on the left. Thus, in a small neighborhood $B$ of $x^* = 0$, we have $\operatorname{dist}(0, \partial \phi(x)) = 1$ for all $x \in B \setminus \Omega$. USC holds with $a = 1$.

This example illustrates two points. First, USC applies directly to nonsmooth constraints without requiring any preliminary reformulation. Second, even when a local reformulation is possible, USC provides a unified criterion that works uniformly for all types of constraints, eliminating the need for case-by-case analysis.

\subsection{Example: Nonsmooth Constraint (USC Works, Classical Fails)}

The following example illustrates a case where the nonsmoothness of the constraint is essential: unlike the previous example, no local reformulation of the constraint can make classical MFCQ applicable.

Consider the problem:
\begin{equation}
    \min_{x \in \mathbb{R}} f(x) = -x
\end{equation}
subject to
\begin{equation}
    g(x) = |x| \le 0.
\end{equation}

The feasible set is $\Omega = \{0\}$. The minimum of $f(x) = -x$ on $\Omega$ is attained at $x^* = 0$. At this point, the constraint is active: $g(0) = 0$. The function $g$ is not differentiable at $x^* = 0$, so the classical MFCQ is not defined.

Note that even if the constraint $|x| \le 0$ is rewritten equivalently as the system $x \le 0$, $-x \le 0$, the classical MFCQ still fails, since there is no direction $d$ such that both $d < 0$ and $-d < 0$ simultaneously. Thus, the failure of regularity is not an artifact of the nonsmooth representation, but reflects a genuine property of the feasible set.

However, USC holds. The penalty function is
\begin{equation}
    \phi(x) = \max\{0, |x|\} = |x|.
\end{equation}
For $x \neq 0$, we have $\partial \phi(x) = \{\operatorname{sign}(x)\}$, so $\operatorname{dist}(0, \partial \phi(x)) = 1$ for all $x \neq 0$. Thus USC holds with $a = 1$ in any neighborhood of $x^* = 0$.

This example demonstrates that USC extends naturally to nonsmooth problems where classical constraint qualifications are not defined, and captures the essential regularity condition independently of how the constraints are formulated.

\subsection{Example: USC Fails when MFCQ Fails}

Consider the problem:
\begin{equation}
    \min_{x \in \mathbb{R}} f(x) = x
\end{equation}
subject to
\begin{equation}
    g_1(x) = x^2 \le 0, \quad g_2(x) = -x^2 \le 0.
\end{equation}

The feasible set is $\Omega = \{0\}$. At $x^* = 0$, both constraints are active, and their gradients are $\nabla g_1(0) = 0$ and $\nabla g_2(0) = 0$. Hence MFCQ fails: there is no direction $d$ such that $0 \cdot d < 0$ and $0 \cdot d < 0$ simultaneously.

USC also fails. The penalty function is $\phi(x) = \max\{0, x^2, -x^2\} = x^2$. For $x \neq 0$, we have $\partial \phi(x) = \{2x\}$, so $\operatorname{dist}(0, \partial \phi(x)) = 2|x| \to 0$ as $x \to 0$. Thus there is no uniform separation of the origin from the subdifferential in any neighborhood of $x^* = 0$.

This example demonstrates that USC detects degeneracy of the gradients, and its failure is consistent with the failure of classical constraint qualifications.

\section{Numerical Verification of USC}

For computational purposes, verifying USC reduces to solving a simple convex optimization problem at each point of interest:
\begin{equation}
    \min_{v \in \partial \phi(x)} \|v\|_2^2.
    \label{eq:verify_usc}
\end{equation}
This is a quadratic programming problem with linear constraints, since $\partial \phi(x)$ is a convex hull of a finite set of vectors:
\begin{equation}
    \min_{\alpha \in \mathbb{R}^r} \left\| \sum_{i=1}^r \alpha_i v_i \right\|_2^2
    \quad \text{s.t.} \quad
    \sum_{i=1}^r \alpha_i = 1, \quad \alpha_i \ge 0,
\end{equation}
where $\{v_1, \dots, v_r\}$ are the active gradients.

In practice, verifying USC does not require checking every point in the neighborhood. Instead, one typically monitors the quantity $d(x) = \operatorname{dist}(0, \partial \phi(x))$ along the sequence of iterates generated by a numerical algorithm. If $d(x_k)$ remains bounded below by a positive constant along the sequence, this provides numerical evidence that USC holds locally. Conversely, if $d(x_k) \to 0$, this indicates a potential violation of USC.

\begin{remark}
A key practical advantage of USC over classical constraint qualifications is that it can be monitored during the optimization process. Classical conditions such as MFCQ or LICQ are defined only at the limit point $x^*$, which is not known a priori in numerical methods. In contrast, USC can be checked at each iterate $x_k$ that is not feasible ($x_k \notin \Omega$) using the same subdifferential information that is already available in nonsmooth optimization algorithms. Monitoring the sequence of iterates $x_k \to x^*$ is sufficient for practical purposes: if the distance $d(x_k) = \operatorname{dist}(0, \partial \phi(x_k))$ remains bounded away from zero along the sequence of infeasible iterates, this provides numerical evidence that USC holds. If $d(x_k) \to 0$, this signals a potential violation of USC, indicating that the problem may be degenerate and that exact penalization with a finite parameter may not be possible.
\end{remark}

\begin{algorithm}
\caption{Verifying USC Along a Sequence}
\label{alg:verify_usc}
\begin{algorithmic}
\STATE Given a sequence of iterates $\{x_k\}$ converging to $x^*$.
\FOR{each $x_k$}
    \IF{$x_k \notin \Omega$}
        \STATE Compute the active set $R(x_k)$.
        \STATE Solve the QP \eqref{eq:verify_usc} to obtain $d_k = \operatorname{dist}(0, \partial \phi(x_k))$.
        \STATE Monitor whether $d_k$ remains bounded below by $a > 0$.
    \ENDIF
\ENDFOR
\STATE If $\displaystyle\liminf_{k \to \infty} d_k \ge a > 0$, USC is numerically verified.
\STATE If $d_k \to 0$, USC is violated.
\end{algorithmic}
\end{algorithm}

\section{Discussion and Conclusions}

The theory of exact penalty functions, developed by Demyanov \cite{Dem2005}, provides sufficient conditions for the existence of an exact penalty constant. A key condition in this theory requires that the constraint violation function has a uniformly negative rate of decrease in a neighborhood of the feasible set. This condition is equivalently expressed as the uniform separation of the origin from the subdifferential of the constraint violation function.

In this paper, we have identified this condition as a general constraint qualification, which we call the Unified Separation Condition (USC). We have established the following results.

First, we proved that MFCQ implies USC for inequality constraints. The proof, based on Gordan's theorem, shows how the neighborhood and the separation constant are obtained from the direction of descent guaranteed by MFCQ. This demonstrates that the classical MFCQ is a special case of USC. Conversely, USC is strictly more general: it applies to nonsmooth problems where MFCQ is not defined, and it may hold even when MFCQ fails.

Second, we proved that LICQ implies USC for equality constraints. The proof uses the linear independence of the active gradients to establish a uniform lower bound on the distance from the origin to the subdifferential. Thus, LICQ is also a special case of USC.

Third, we demonstrated that USC extends naturally to nonsmooth problems where classical constraint qualifications are not defined. Unlike MFCQ and LICQ, USC does not require differentiability of the constraint functions.

Fourth, we introduced a local version of USC that can be applied in a neighborhood of a solution. This local nature makes USC practical for verifying optimality conditions and analyzing numerical algorithms.

Fifth, we discussed the computational verification of USC via convex quadratic programming. Since the subdifferential of the constraint violation function is a convex hull of a finite set of gradients, verifying USC reduces to solving a simple quadratic programming problem. This makes USC computationally tractable.

The USC offers a unified framework that bridges classical smooth constraint qualifications and nonsmooth optimization. It provides a single condition that applies to all types of constraints, does not require differentiability, arises naturally from exact penalty theory, and is computationally verifiable.

Future research directions include a more systematic study of USC in infinite-dimensional settings, particularly for optimal control problems with state and mixed constraints, as well as variational inequalities. The relationship between USC and second-order optimality conditions, as well as stability analysis under perturbations of the constraints, also warrants further investigation. Finally, the design of efficient numerical methods for verifying USC in discretized problems remains an open and practically important problem.

\section*{Acknowledgments}

The author acknowledges the use of AI-assisted language tools for text editing. All scientific content, including the problem formulation, theoretical results, proofs, and analysis, is the original work of the author. 

The author wishes to honor the memory of Prof. V. F. Demyanov, whose foundational work on exact penalty theory inspired this research.

\end{document}